\documentclass[12pt]{amsart}
\usepackage[all]{xy} 

\def\bz{{\mathbb Z}\,}
\def\bq{{\mathbb Q}}
\def\bg{{\mathbb G}}
\def\spec{{\rm{Spec}}\,}
\def\fb{\overline{F}}
\def\ut{\widetilde{U}}
\def\cok{{\rm{Coker}}}

\def\tto{\tt^{\e\circ}}
\def\krn{\text{Ker}\e}
\def\sto{{\s T}^{\e\circ}}

\def\be{\kern -.1em}
\def\lbe{\kern -.05em}
\def\s{\mathcal }
\def\ra{\rightarrow}
\def\e{\kern 0.08em}
\def\le{\kern 0.04em}
\def\ng{\kern -0.04em}
\def\tt{\widetilde{\s T}}

\def\fc{\fb\!\phantom{.}^{\lbe *}}

\newtheorem{theorem}{Main Theorem\!\!}

\newtheorem{lemma}{Lemma}[section]
\newtheorem{teorema}[lemma]{Theorem}
\newtheorem{corollary}[lemma]{Corollary}

\theoremstyle{definition}

\theoremstyle{remark}
\newtheorem{remark}[lemma]{Remark}
\newtheorem{remarks}[lemma]{Remarks}

\begin{document}

\title[Chevalley's formula for tori]{Chevalley's ambiguous class
number formula for an arbitrary torus}


\author{Cristian D. Gonz\'alez-Avil\'es}
\address{Departamento de Matem\'aticas, Universidad Andr\'es Bello,
Santiago, Chile} \email{cristiangonzalez@unab.cl}


\thanks{The author is partially supported by Fondecyt grant
1061209 and Universidad Andr\'es Bello grant DI-29-05/R}

\maketitle

\begin{abstract} In this paper we obtain Chevalley's ambiguous
class number formula for an arbitrary torus $T$ over a global field.
The classical formula of C.Chevalley may be recovered by setting
$T=\bg_{m}$ in our formula. A key ingredient of the proof is the
work of X.Xarles on groups of components of N\'eron-Raynaud models
of tori.
\end{abstract}

\section{Introduction}

Let $F$ be a global field and let $S$ be a nonempty finite set of
primes of $F$ containing the archimedean primes in the number field
case. Further, let $C_{F,\e S}$ denote the $S$-ideal class group of
$F$ and let $K/F$ be a finite Galois extension with Galois group
$G$. We will write $C_{K,\e S}$ for the $S_{K}$-ideal class group of
$K$, where $S_{K}$ is the set of primes of $K$ lying above the
primes in $S$. Now let $\s O_{K,S}$ denote the ring of
$S_{K}$-integers of $K$ and, for each prime $v$ of $F$, let $e_{v}$
denote the ramification index of $v$ in $K/F$. The order of a finite
set $M$ will be denoted by $[M]$.

In his 1933 thesis [5], C.Chevalley obtained the following ambiguous
class number formula, which has become well-known.

\begin{teorema} We have

$$
\frac{\left[\e C_{K,\e S}^{\e G}\e\right]}{ \left[\e
C_{F,S}\e\right]}=\frac{\left[\e H^{1}\be\big(G, K^{*}/\e\s O_{K,\e
S}^{\e *}\e\big)\e\right]}{\left[\e H^{1}\ng\big(G,\s O_{K,\e S}^{\e
*}\big)\e\right]}\displaystyle\prod_{v\notin S} e_{v}.
$$
\end{teorema}

In this paper we extend the above formula to an arbitrary torus $T$
over $F$ (the preceding formula may then be recovered from our Main
Theorem below by setting $T=\bg_{m,F}$ there).

Let $\ut=\spec {\s O}_{K,\e S}$, let $T$ be an arbitrary torus over
$F$ and let $\tto$ denote the identity component of the
N\'eron-Raynaud model of $T_{K}$ over $\ut$. Then $G$ acts on the
class group $C_{\e T,K,S}=C\be\big(\e\tt^{\circ}\e\big)$. Now, for
each prime $v\notin S$, we fix once and for all a prime $w=w_{v}$ of
$K$ lying above $v$ and write $\kappa(v)$ (resp. $\kappa(w)$) for
the residue field of $F$ (resp. $K$) at $v$ (resp. $w$). Let $G_{w}$
be the decomposition group of $w$ in $K/F$. For each $v\notin S$, we
will write $d_{v}$ for the dimension of the largest split subtorus
of $T_{F_{v}}$, where $F_{v}$ is the completion of $F$ at $v$. Now
let $\Phi_{w}$ denote the scheme of connected components of the
reduction of $\tt$ at $w$. There exist maps
$$
H^{1}\ng\big(G,T(K)\big/\e\tto\be\big(\ut\big)\big)
\ra\displaystyle\bigoplus_{v\notin S}
H^{1}\!\left(G_{w_{v}},\Phi_{w_{v}}\be(\kappa(\lbe
w_{v}\lbe))\right)
$$
and
$$
H^{1}\ng\big(G,\tto\be(\ut)\big)\ra H^{1}(G,T(K)).
$$

\smallskip

\noindent We denote their kernels by $H^{1}\be\big(G,
T(K)/\e\tto\be\big(\ut\big)\big)^{\e\prime}$ and $H^{\e
1}\ng\big(G,\tto\be(\ut)\big)^{\e\prime}$, respectively. Finally,
let $U=\spec\s O_{F,S}$ and let ${\s T}^{\le\circ}$ denote the
identity component of the N\'eron-Raynaud model of $T$ over $U$.

\begin{theorem} We have

$$
\frac{\left[\e C_{\e T,K,S}^{G}\e\right]}{ \left[\e C_{\e
T,F,S}\e\right]}=\frac{\left[\e H^{1}\be\big(G,
T(K)/\e\tto\be\big(\ut\big)\big)^{\e\prime}\e\right]\displaystyle
\prod_{\e v\notin S}\, e_{v}^{d_{v}}\ell_{v}}{ \left[
H^{1}\ng\big(G,\tto\be(\ut)\big)^{\e\prime}\e\right]\left[\e
\tto\lbe\big(\ut\e\big)^{G}\be\colon\be {\s
T}^{\circ}\be(U)\e\right]^{-1}}\,,
$$
where the $\ell_{v}$ are certain local factors described in Lemma
3.5.
\end{theorem}

We should remark that some authors have discussed the {\it norm map}
for the trivial torus $\bg_{\be F}$ (see Section 5). However, the
capitulation map for arbitrary tori has not been discussed
previously. In Section 4 we specialize our main result to the case
of norm tori and obtain some interesting new formulas.

\section*{Acknowledgements}

I am very grateful to D.Lorenzini for several helpful comments and
for encouraging me to write Section 4. I also thank the referee for
suggesting several ways to improve the original version of this
work. Finally, I thank B.Kunyavskii and X.Xarles for sending me
copies of [10] and [15].

\section{Preliminaries}

Let $F$ be a global field, i.e. $F$ is a finite extension of ${\Bbb
Q}$ (the ``number field case") or is finitely generated and of
transcendence degree 1 over a finite field of constants (the
``function field case"). Let $S$ be a nonempty finite set of primes
of $F$ containing the archimedean primes in the number field case
and let ${\s O}_{F,\e S}$ denote the ring of $S$-integers of $F$. We
will write $U=\spec {\s O}_{F,\e S}$. Further, for any
non-archimedean prime $v$ of $F$, $F_{\be v}$ will denote the
completion of $F$ at $v$ and ${\s O}_{v}$ will denote the ring of
integers of $F_{\be v}$. The residue field of ${\s O}_{v}$ will be
denoted by $\kappa(v)$. Further, we will write $j\colon\spec F\ra U$
for the inclusion of the generic point of $U$ and, for each closed
point $v\in U$, $i_{v}\colon\spec\kappa(v)\ra U$ will denote the
similar map at $v$.

\smallskip

Let $K/F$ be a finite Galois extension of global fields with Galois
group $G$. We fix a separable algebraic closure $\fb$ of $F$
containing $K$ and write $G_{K}$ for $\text{Gal}\big(\e\fb/K\big)$.
The set of primes of $K$ lying above the primes in $S$ will be
denoted by $S_{K}$. For each prime $v$ of $F$, we fix once and for
all a prime $w$ of $K$ lying above $v$ and write $G_{w}$ for the
decomposition group of $w$ in $K/F$. When necessary, we will write
$w_{v}$ for $w$. The inertia subgroup of $G_{w}$ will be denoted by
$I_{w}$. Now, for each $v$ as above, we fix a prime
$\overline{w}=\overline{w_{v}}$ of $\fb$ lying above $w=w_{v}$. Then
$\fb_{\be\overline{w}}$ is a separable algebraic closure of $F_{v}$
containing $K_{w}$. We
$I_{\e\overline{w}}=\text{Gal}\big(\e\fb_{\be\overline{w}}/
K_{w}^{\text{nr}}\e\big)$ and
$I_{\e\overline{v}}=\text{Gal}\big(\e\fb_{\be\overline{w}}/
F_{v}^{\text{nr}}\e\big)$, where $K_{w}^{\text{nr}}$ (resp.
$F_{v}^{\text{nr}}$) is the maximal unramified extension of $K_{w}$
(resp. $F_{v}$) inside $\fb_{\be\overline{w}}$. Clearly,
$I_{\e\overline{w}}$ is a subgroup of $I_{\e\overline{v}}$ and there
exist natural isomorphisms
$I_{\e\overline{v}}/I_{\e\overline{w}}=\text{Gal}(K_{w}^{\text{nr}}/
F_{v}^{\text{nr}})=I_{w}$. We will write $G(w)$ (or $G(w_{v})$, if
necessary) for $\text{Gal}(\kappa(w)/\kappa(v))$, which we will
identify with the quotient $G_{w}/I_{w}$. Further, we will write
$e_{v}$ for the ramification index of $v$ in $K$.

\smallskip

Now, for any finite set of primes $S^{\e\prime}$ of $F$ containing
$S$, set
$$
{\Bbb A}_{ F,S^{\prime}}=\prod_{v\in
S^{\prime}}F_{v}\times\prod_{v\notin S^{\prime}}{\s O}_{v}.
$$
Then the ring of adeles of $F$ is by definition
$$
{\Bbb A}_{F}=\varinjlim_{S^{\prime}\supset S} {\Bbb A}_{
F,S^{\prime}}.
$$
Let $\s H$ be a $U$-group scheme of {\it finite type} with smooth
generic fiber. The {\it class set of $\,\s H$} is by definition the
double coset space
$$
C(\s H)=\s H(\lbe F\lbe)\be\backslash\e \s H\!\left(\lbe{\Bbb
A}_{F}\lbe\right)\be/\e\s H\!\left(\lbe{\Bbb A}_{F,S}\lbe\right).
$$
The cardinality of this set, when finite (which is the case if $\s
H_{F}$ is a torus over $F$ [6, \S 3]), is called the {\it class
number of $\s H$}. If $\s H=\bg_{m,U}$, $C(\s H)$ may be identified
with the $S$-ideal class group of $F$.

\smallskip

Let ${\s O}_{K,\e S}$ be the ring of $S_{K}$-integers of $K$, let
$\ut=\spec {\s O}_{K,\e S}$ and let $\widetilde{\jmath}\e\colon
\spec K\ra\ut$ be the inclusion of the generic point. Now let $T$ be
an $F$-torus and let $\tt$ denote the N\'eron-Raynaud model of
$T_{K}$ over $\ut$. Thus $\tt$ is a smooth and separated $\ut$-group
scheme which is {\it locally} of finite type and represents the
sheaf $\widetilde{\jmath}_{\e *}T_{K}$ on $\ut_{\text{sm}}$. See [2,
\S 10] for more details. The N\'eron-Raynaud model of $T$ over $U$
will be denoted by $\s T$. Let $\tto$ (resp. $\s T^{\le\circ}$)
denote the (fiberwise) identity component of $\tt$ (resp. $\s T$).
Then $\tto$ is an affine [9, Proposition 3, p.18]\footnote{ The
proof of this proposition is valid in the function field case as
well, as pointed out by D.Lorenzini.} {\it smooth} $U$-group scheme
of {\it finite type}. For each prime $w\notin S_{K}$, let
$\Phi_{w}=i_{w}^{*}\be\big(\le\tt/\tto\le\big)$ denote the sheaf of
connected components of $\tt$ at $w$. Then $\Phi_{w}$ is represented
by an \'etale $\kappa(w)$-scheme of finite type and hence completely
determined by the $G_{\lbe\kappa(w)}$-module
$\Phi_{w}\be\big(\e\overline{\kappa\small(w\small)}\e\big)$. For
each prime $v\notin S$, $\Phi_{v}$ will denote the sheaf of
connected components of $\s T$ at $v$. If $v\notin S$ and $w=w_{v}$
is the prime of $K$ lying above $v$ fixed previously, there exists a
canonical capitulation map
\begin{equation}
\delta_{v}\colon\Phi_{v}(\kappa(v))\ra
\Phi_{w}\be\big(\kappa(w))^{G(w)}
\end{equation}

We will need the {\it Nisnevich topology}. It is a Grothendieck
topology on $U$ stronger than the Zariski topology but weaker than
the \'etale topology. It was introduced in [16] in order to
generalize the well-known cohomological interpretation
$C_{F,S}=H^{\e 1}_{\text{\'et}}(U,\bg_{m})$ of the $S$-ideal class
group of $F$ to arbitrary, generically smooth, $U$-group schemes $\s
H$ of finite type. That is, the following holds:
\begin{equation}
C(\s H)=H^{\e 1}_{\text{Nis}}(U,\s H).
\end{equation}
See [17, esp. pp.281-289] for a partial account of the above
theory\footnote{ Readers wishing to learn more about the non
$K$-theoretic applications of the Nisnevich topology, and who are
not prepared to wait for the publication of a proper survey, are
advised to obtain a copy of [16].}. In this paper, the $S$-class
group $C_{\e T,F,S}$ of $T$ over $F$ is defined to be
$C\lbe\big(\e\s T^{\e\circ}\lbe\big)$ (as noted above, $\s
T^{\e\circ}$ is of finite type and not only generically smooth but
in fact smooth). Its cardinality will be denoted by $h_{\e T,F,S}$.

\begin{remark} One can also define a class group
$C_{\e T,F,S}=C(\s H)$ for some other integral model of finite type
$\s H$ of $T$. We choose the model $\sto$ because, with this choice,
the well-known exact sequence
$$
1\ra\s O_{K,S}^{\e *}\ra K^{*}\ra\s I_{K,S}\ra C_{K,S}\ra 0,
$$
where $\s I_{K,S}$ denotes the group of fractional $S$-ideals of
$K$, admits the natural generalization (5) below (this will play a
fundamental role in [8]). Assume now that $F$ is a number field.
V.Voskresenskii [21, \S 20] and B.Kunyavskii et al. [10, p.47] have
defined (for $F$ a number field and $S$ equal to the set of
archimedean primes of $F$) a class group of $T$ using a certain
``standard model" $\s H=\s T_{0}$ of $T$ which (unlike the
N\'eron-Raynaud model $\s T$ of $T$) is of finite type (but can be
non-smooth). The corresponding class number $h(\s T_{0})$ is minimal
among all class numbers $h(\s T_{1})$ as $\s T_{1}$ runs over the
family of all integral models of $T$ of finite type (see [21,
p.197]). Since $\sto$ is such a model (see above), one has $h(\s
T_{0})\!\mid\! h_{\e T,F,\infty}$. When $T$ splits over a {\it
tamely ramified} extension of $F$, the identity components of $\s T$
and $\s T_{0}$ are canonically isomorphic [10, Theorem 3, p.49]. In
this case, therefore, if $\s T_{0}$ has connected fibers (which is
the case, for example, if $T$ is a quasi-split $F$-torus of the form
$T=R_{K/F}(\bg_{m,K}^{\le d})$, where $K/F$ is a tamely ramified
Galois extension of number fields [10, p.48]), then $h(\s T_{0})$
and $h_{\e T,F,\infty}$ coincide.
\end{remark}

\smallskip

We will write $X={\rm{Hom}}\!\left(\e
T\!\left(\e\fb\e\right),\fc\right)$ for the $G_{\lbe F}$-module of
characters of $T$. For any $G_{\be F}$-lattice $Y$, we will write
$Y^{\le\vee}$ for the $\bz$-linear dual of $Y$, i.e.,
$Y^{\le\vee}=\text{Hom}_{\e\bz}\be(Y,\bz)$.

\smallskip

If $f\colon A\ra B$ is a homomorphism of abelian groups with finite
kernel and cokernel, we define
$$
q(f)=\frac{[\cok\e f\e]}{[{\rm{Ker}}\e f\e]}.
$$
This function is multiplicative on short exact sequences, in the
sense that if $f^{\le\bullet}\colon X^{\bullet}\ra Y^{\bullet}$ is a
map of short exact sequences and $q(f^{r})$ is defined for
$r=1,2,3$, then $q(f^{2})=q(f^{1})q(f^{3})$.

\smallskip

If $M$ is any abelian group and $m$ is any positive integer, $M_{m}$
will denote its $m$-torsion subgroup and we will write $M/m$ for
$M/mM$. Applying $\text{Tor}^{\e\bz}(\bz/m,-)$ to a given exact
sequence $0\ra A\ra B\ra C\ra 0$ of abelian groups, we obtain an
exact sequence
\begin{equation}
0\ra A_{m}\ra B_{m}\ra C_{m}\overset{c}\longrightarrow A/m\ra B/m\ra
C/m\ra 0,
\end{equation}
where $c$ is a certain ``connecting homomorphism". See, e.g., [22,
\S 3.1.1, p.66]. The (Pontryagin) dual of $M$ is by definition
$M^{D}=\text{Hom}(M,\bq/\bz)$. If $G$ is a group and $M$ is a (left)
$G$-module, $M^{D}$ will be endowed with its natural $G$-action
(see, e.g., [1, p.94]). Further, $M_{G}$ (resp. $M^{G}$) will denote
the largest quotient (resp. subgroup) of $M$ on which $G$ acts
trivially.  It is not difficult to check that $(M^{D})^{\e
G}=(M_{G})^{D}$. Further, if $M$ is finite and $G$ is pro-cyclic,
then $[M_{G}]=\left[\e M^{\le G}\le\right]$ (see, e.g., [1, proof of
Proposition 11, p.109]). It follows that $\left[\e(M^{D})^{\le
G}\e\right]=[M^{\le G}]$ (we will use the latter fact often when
referring to [23, Corollary 2.18, p.175]). We also note that, if $H$
is a normal subgroup of $G$, then there exists a canonical
isomorphism $\bz\be[G]^{H}=\bz\be[\le G\lbe/\lbe H\le]$. In addition
to the above, we will need the following well-known facts: if
$M=\bz\be[G]\otimes Y$ for some abelian group $Y\,$\footnote{ All
tensor products in this paper are taken over $\bz$.}, then the Tate
groups $\widehat{H}_{\e 0}(G,M)=\widehat{H}^{\e 0}(G,M)=0$ (see,
e.g., [1, Proposition 6, p.102]). If $Y$ is a torsion-free abelian
group, then $\text{Tor}_{1}^{\e\bz}(Y,-)=0$. Further, if $G$ is a
torsion group which acts trivially on $Y$, then $H^{\e
1}(G,Y)=\text{Hom}(G,Y)=0$.

\section{Proof of the main theorem}

There exists a natural exact sequence of sheaves on
$\ut_{\text{sm}}$:
$$
0\ra\tto\ra\tt\ra\bigoplus_{w\notin S_{K}}
\,\,(i_{w})_{*}\Phi_{w}\ra 0.
$$
The preceding exact sequence induces an exact sequence
\begin{equation}
1\ra\tto\be\big(\ut\big)\ra T(K)\ra\bigoplus_{w\notin
S_{K}}\,\Phi_{w}(\kappa(w))\ra
H^{1}_{{\rm{\acute{e}t}}}\big(\e\ut,\tto\big)\ra
H^{1}_{{\rm{\acute{e}t}}}\big(\e\ut,\tt\e\big).
\end{equation}

\begin{lemma} There exists a canonical isomorphism
$$
{\rm{Ker}}\ng\left[\e
H^{1}_{{\rm{\acute{e}t}}}\big(\e\ut,\tto\big)\ra
H^{1}_{{\rm{\acute{e}t}}}\big(\e\ut,\tt\e\big)\e\right]=H^{1}_{{\rm{Nis}}}(\ut,
\tto)=C_{\e T,K,S}.
$$
\end{lemma}
\begin{proof} Since $\tt=\widetilde{\jmath}_{\e *}T_{K}$ as
\'etale sheaves, the Cartan-Leray spectral sequence
$$
H^{p}_{\text{\'et}}\big(\e\ut,R^{\e q}\widetilde{\jmath}_{\e
*}T_{K}\big)\implies H^{p+q}_{\text{\'et}}(K,T)
$$
yields an injection $H^{1}_{\text{\'et}}\big(\e\ut,\tt\e\big)=
H^{1}_{\text{\'et}}\big(\e\ut,\widetilde{\jmath}_{\e
*}T_{K}\e\big)\hookrightarrow H^{1}_{\text{\'et}}(K,T\e)$. Thus
$H^{1}_{\text{\'et}}\big(\e\ut,\tt\e\big)$ may be replaced with
$H^{1}(K,T\e)$ in the statement of the lemma. Now [17, 1.44.2,
p.286] completes the proof.
\end{proof}

By the lemma and (2), (4) induces a canonical exact sequence of
$G$-modules
\begin{equation}
1\ra\tto\be\big(\ut\big)\ra T(K)\ra\displaystyle\bigoplus_{v\notin
S}\bigoplus_{w\mid v}\,\Phi_{w}(\kappa(w))\ra  C_{\e T,K,S}\ra 0.
\end{equation}
We will identify the $G$-module $\bigoplus_{w\mid
v}\,\Phi_{w}(\kappa(w))$ with the $G$-module {\it coinduced} by the
$G_{w_{v}}$-module $\Phi_{w_{v}}(\kappa(w_{v}))$. Thus, by Shapiro's
lemma,
$$
H^{\e i}\!\big(G,\textstyle\bigoplus_{w\mid
v}\,\Phi_{w}(\kappa(w))\big)=H^{\e
i}(G_{w_{v}},\Phi_{w_{v}}\be(\kappa(w_{v})))
$$
for every $i\geq 0$. Note also that, since $I_{w_{v}}$ acts
trivially on $\Phi_{w_{v}}\be(\kappa(w_{v}))$,
$$
\Phi_{w_{v}}\be(\kappa(w_{v}))^{G_{w_{v}}}=\Phi_{w_{v}}\be(\kappa(w_{v}))
^{G(w_{v})}.
$$

Clearly, (5) induces an exact sequence
$$
C_{\e T,K,S}^{\e G}\overset{\alpha}\longrightarrow
H^{1}\ng\big(G,T(K)\big/\e\tto\be\big(\ut\big)\big)
\overset{\beta}\longrightarrow\displaystyle\bigoplus_{v\notin S}
H^{1}\!\left(G_{w_{v}},\Phi_{w_{v}}\be(\kappa(w_{v}))\right).
$$
Define
$$
(C_{\e T,K,S}^{\e G})_{{\text{trans}}}=\krn\e(\alpha)
$$
and
\begin{equation}
H^{1}\be\big(G, T(K)/\e\tto\be\big(\ut\big)\big)^{\e\prime}=
{\rm{Ker}}\e(\beta).
\end{equation}
Then
\begin{equation}
\left[\e C_{\e T,K,S}^{\e G}\e\right]= \left[\e (C_{\e T,K,S}^{\e
G})_{{\text{trans}}}\e\right] \left[\e H^{1}\be\big(G,
T(K)/\e\tto\be\big(\ut\big)\big)^{\e\prime}\e\right].
\end{equation}

\begin{remark} If $T$ splits over $K$, i.e., $T_{K}=\bg_{m,K}^{\le d}$
for some $d$, then
$H^{1}\!\left(G_{w_{v}},\Phi_{w_{v}}\be(\kappa(w_{v}))\right)=
\text{Hom}\!\left(G_{w_{v}},\bz^{\! d}\le\right)=0$. In this case,
therefore, $H^{1}\be\big(G,
T(K)/\e\tto\be\big(\ut\big)\big)^{\e\prime}=H^{1}\be\big(G,
T(K)/\e\tto\be\big(\ut\big)\big)$.
\end{remark}

Now there exists a canonical {\it capitulation map}
$$
j_{\e T, K/F}\colon C_{\e T,F,S}\ra C_{\e T,K,S}^{\e G}
$$
which is the composite
$$
H^{\e 1}_{\text{Nis}}(U,\sto)\overset{\text{ad}}\ra H^{\e
1}_{\text{Nis}}\be\big(\e
U,\pi_{*}\pi^{*}\e\sto\e\big)\hookrightarrow H^{\e
1}_{\text{Nis}}\be\big(\e\ut,\pi^{*}\sto\big)^{G}\overset{\text{bc}}
\ra H^{\e 1}_{\text{Nis}}\be\big(\ut,\tto\big)^{G},
$$
where $\pi\colon\ut\ra U$ is the canonical map, ``ad" is induced by
the adjoint morphism $\sto\ra\pi_{*}\pi^{*}\e\sto$, the middle
injection is the first nontrivial map in the exact sequence of terms
of low degree belonging to the Cartan-Leray spectral sequence $H^{\e
p}_{\text{Nis}}(U,R^{\e q}\pi_{*}\pi^{*}\sto)\!\!\implies\!\!
H^{p+q}_{\text{Nis}}\be\big(\e\ut,\pi^{*}\sto\big)$ [17, 1.22.1,
p.270] and the map ``bc" is induced by the base change morphism
$\pi^{*}\sto=\sto\times_{U}\e\ut\ra\tto$ [2, \S 7.2, Theorem 1(i),
p.176]. It is not difficult to check that the image of $j_{\e
T,K/F}$ is contained in $(C_{\e T,K,S}^{\e G})_{{\text{trans}}}$. We
will write $j_{\e T,K/F}^{\e\prime}\colon C_{\e T,F,S}\ra (C_{\e
T,K,S}^{\e G})_{{\text{trans}}}$ for the induced map. Then (7) may
be rewritten as
\begin{equation}
\frac{\left[\e C_{\e T,K,S}^{\e G}\e\right]}{ \left[\e
C_{T,F,S}\e\right]}=q(\e j_{\e T,K/F}^{\e\prime}\e) \left[\e
H^{1}\be\big(G,
T(K)/\e\tto\be\big(\ut\big)\big)^{\e\prime}\e\right].
\end{equation}
Now there exists a natural exact commutative diagram

\begin{equation}
\xymatrix{0\ar[r] & T(F)/\e{\s T}^{\e\circ}(U)\ar[d]^{\gamma}\ar[r]
& \displaystyle\bigoplus_{v\notin
S}\Phi_{v}(\kappa(v))\ar@{->>}[r]\ar[d]^{\bigoplus_{v\notin
S}\delta_{v}} &
C_{\e T,F,S}\ar[d]^{j_{\e T,K/F}^{\e\prime}}\\
0\ar[r] &\big(T(K)/\e\tto\be\big(\ut\big)\big)^{G}\ar[r] &
\displaystyle\bigoplus_{v\notin
S}\Phi_{w_{v}}\be(\kappa(w_{v}))^{G(w_{v})}\ar@{->>}[r] & (C_{\e
T,K,S}^{\e G})_{{\text{trans}}}}
\end{equation}
where, for each $v\notin S$, $\delta_{v}$ is the capitulation map
(1) and $\gamma$ is induced by $\bigoplus_{\e v\notin S}\delta_{v}$.

\smallskip

\begin{lemma} For each $v\notin S$, there exists an isomorphism of
finite groups
$$
{\rm{Ker}}\ng\left[\e\Phi_{v}(\kappa(v))\overset{\delta_{v}
}\longrightarrow\Phi_{w_{v}}\be(\kappa(w_{v}))^{G(w_{v})}\e\right]=
H^{1}\!\left(I_{w_{v}},T\!\left(K_{w_{v}}
^{\e{\rm{nr}}}\right)\right)^{G_{\kappa(v)}}.
$$
Here $I_{w_{v}}$ is identified with
${\rm{Gal}}(K_{w_{v}}^{\text{nr}}/ F_{v}^{\text{nr}})$. In
particular, if $v$ is unramified in $K/F$, then $\delta_{v}$ is
injective.
\end{lemma}
\begin{proof} Set $w=w_{v}$. Since $\text{Ker}\,\delta_{v}$ is a
torsion group (cf. [7, Theorem 1]), we have $\text{Ker}\,\delta_{v}=
\text{Ker}\,\delta_{v}^{\e\prime}$, where
$\delta_{v}^{\e\prime}\colon \Phi_{v}(k(v))_{\e\text{tors}}\ra
\Phi_{w}(\kappa(w))^{G(w)}_{\e\text{tors}}$ is the induced map.
Since $\Phi_{v}(k(v))_{\e\text{tors}}$ is finite by [23, Corollary
2.18], we conclude that $\text{Ker}\,\delta_{v}$ is finite as well.
Now, by [23, Corollary 2.18] and [13, Corollary I.2.4, p.35], there
exists a natural exact commutative diagram\footnote{ The
commutativity of this diagram follows by examining the proofs of
Lemma 2.13, Proposition 2.14, Lemma 2.17 and Corollary 2.18 in [23].
The key point is that the map $H^{1}(I_{\overline{w}},X)\ra
H^{1}(I_{\overline{v}},X)$ (notation as in [op.cit.]) which
corresponds to the map $\text{Ext}_{\bz}^{\e
1}(\Phi_{w},\bz)\ra\text{Ext}_{\bz}^{\e 1}(\Phi_{v},\bz)$ under the
isomorphism of [op.cit., Proposition 2.14] is induced by the norm
map $N_{I_{w}}\colon X^{I_{\overline{w}}}\ra X^{I_{\overline{v}}}$
(see, especially, the diagram at the bottom of p.174 of [op.cit.]).}
\[
\xymatrix{\Phi_{v}(\kappa(v))_{\e\text{tors}}\ar[d]^{\delta^{\e\prime}
_{\e v}}\ar[r]^{\sim}&
H^{1}(I_{\overline{v}},T)^{\e G_{k(v)}}\ar[d]^{\text{Res}}\\
\Phi_{w}(\kappa(w))^{G(w)}_{\e\text{tors}}\ar[r]^{\sim}&
H^{1}(I_{\overline{w}},T)^{\e G_{k(v)}}, }
\]
where the right-hand vertical map is induced by the restriction map
$H^{1}(I_{\overline{v}},T)\ra H^{1}(I_{\overline{w}},T)$. The result
is now immediate from the inflation-restriction exact sequence.
\end{proof}

An $F$-torus $T$ is said to have {\it good reduction} at a
nonarchimedean prime $v$ of $F$ if ${\s T}^{\e\circ}_{{\s O}_{v}}$
is a torus over $\s O_{v}$. The following are equivalent conditions
(see [15, (1.1), p.462]): (a) $T$ has good reduction at $v$ as
defined above; (b) ${\s T}^{\e\circ}_{\kappa(v)}$ is a torus over
$\kappa(v)$; (c) $I_{\overline{v}}$ acts trivially on $X$; (d) $T$
splits over an unramified extension of $F_{v}$.

Since only finitely many primes of $F$ can ramify in a splitting
field of $T$, we see that $T$ has good reduction at all but finitely
many primes of $F$.

\begin{lemma} Let $v$ be a nonarchimedean prime of $F$ which is
{\rm{unramified}} in $K/F$ and where $T$ has {\rm{good reduction}}.
Then the capitulation map
$$
\delta_{v}\colon\Phi_{v}(\kappa(v))\ra
\Phi_{w_{v}}\be(\kappa(w_{v}))^{G(w_{v})}
$$
is an isomorphism.
\end{lemma}
\begin{proof} Proposition 3.2 of [3] and functoriality
immediately reduce the proof to the case $T={\Bbb G}_{m,F}$, where
$\delta_{v}\colon\bz\ra\bz$ is multiplication by the ramification
index $e_{v}$ of $v$ in $K/F$. The lemma is clear in this case.
\end{proof}

Lemmas 3.3 and 3.4 show that the map $\bigoplus_{\e v\notin
S}\delta_{v}$ in diagram (9) has a finite kernel and cokernel.
Therefore
$$
q(\e\oplus_{\e v\notin S}\e\delta_{v}\e)=\displaystyle\prod_{v\notin
S}q(\e\delta_{v}\e)
$$
is defined and diagram (9) yields
$$
\displaystyle\prod_{v\notin S}q( \e\delta_{v}\e)=q(\gamma)\e q(
j_{\e T,K/F}^{\e\prime}).
$$
Consequently, (8) yields the identity
\begin{equation}
\frac{\left[\e C_{\e T,K,S}^{\e G}\e\right]}{ \left[\e C_{\e
T,F,S}\e\right]}=\frac{\left[\e H^{1}\be\big(G,
T(K)/\e\tto\be\big(\ut\big)\big)^{\e\prime}\e\right]}{
q(\e\gamma\e)}\cdot\displaystyle\prod_{v\notin S} q(
\e\delta_{v}\e).
\end{equation}

It remains to compute the factors $q(\delta_{v})$ (for $v\notin S$)
and $q(\e\gamma\e)$. We compute first the local factors $q(
\delta_{v})$.

\smallskip

Let $v\notin S$. The inertia group $I_{\overline{v}}$ acts on $X$
through a finite quotient $J_{\overline{v}}$ (say), and we have a
natural map
$$
N_{v}\colon X\ra X^{I_{\overline{v}}},\chi\mapsto\sum_{g\e\in
J_{\overline{v}}}\chi^{\e g}.
$$
Let $\widehat{T}_{\lbe v}$ be the $F_{v}$-torus corresponding to the
subgroup ${\rm{Ker}}\e N_{v}$ of $X$ and let $T^{\e *}_{v}$ be the
largest split subtorus of $T_{F_{v}}$. Let $d_{v}$ be the dimension
of $T^{\e *}_{v}$. Set $w=w_{v}$. Since
$H^{1}(G(w),\bz)=\text{Hom}(G(w),\bz)=0$, Corollary 3.3 of [3]
yields an exact commutative diagram
$$
\xymatrix{0\ar[r]&\Phi_{v}^{*}(\kappa(v))\ar[d]^{\delta_{v}
^{*}}\ar[r]& \Phi_{v}(\kappa(v))\ar[r]\ar[d]^{\delta_{v}}&
\widehat{\Phi}_{v}(\kappa(v))\ar[d]^{\widehat{\delta}_{v}}
\ar[r]& 0\\
0\ar[r] &\Phi_{w}^{*}(\kappa(w))^{\e G(w)}\ar[r] &
\Phi_{w}(\kappa(w))^{\e G(w)}\ar[r] &
\widehat{\Phi}_{w}(\kappa(w))^{\e G(w)}\ar[r]& 0. }
$$
The left-hand vertical map $\delta_{v}^{\le *}$ may be identified
with the map ${\Bbb Z}^{d_{\e v}}\ra{\Bbb Z}^{d_{\e v}},
(n_{i})_{1\leq i\leq d_{\e v}}\mapsto (e_{v}n_{i})_{1\leq i\leq
d_{\e v}}$. Thus
$$
q(\delta_{v}^{\le *})=e_{v}^{\e d_{\le v}}.
$$
On the other hand, by [23, Corollary 2.19(b), p.175] and the fact
that $\big[\le M_{\le G_{\lbe\kappa(v)}}\le\big]=\big[\le M^{\le
G_{\lbe\kappa(v)}}\be\big]$ for any finite $G_{\kappa(v)}$-module
$M$, $\widehat{\Phi}_{v}\be\big(\e\overline{\kappa\small(v\small)}
\e\big)$ and
$\widehat{\Phi}_{w}\be\big(\e\overline{\kappa\small(w\small)}
\e\big)$ are finite $G_{\kappa(v)}$-modules and
$$\begin{array}{rcl}
q\big(\e\widehat{\delta}_{v}\e\big)&=&\left[\e\widehat{\Phi}_{w}
\be\big(\,\overline{\kappa\small(w\small)}\,\big)^{G_{\lbe\kappa(\be
v\be)}}\le\right]\Big/ \left[\e\widehat{\Phi}_{v}
\be\big(\,\overline{\kappa\small(v\small)}\,\big)^{G_{\lbe\kappa(\be
v\be)}}\e\right]\\\\
&=&\big[\e H^{1}(I_{\overline{w}},\,{\rm{Ker}}\e N_{v})^{\e
G_{\be\kappa(\be v\be)}}\e\big]\big/\big[\e
H^{1}(I_{\overline{v}},\,{\rm{Ker}}\e N_{v})^{\e G_{\be\kappa(\be
v\be)}}\e\big].
\end{array}
$$
Thus the following holds.

\begin{lemma} Let $v\notin S$. Then
$$
q(\delta_{v})=e_{\le v}^{\le d_{v}}\be\frac{\left[\e
H^{1}(I_{\overline{w}_{v}},{\rm{Ker}}\e N_{v})^{\e G_{\be\kappa(\be
v\be)}}\e\right]}{\left[\e H^{1}(I_{\overline{v}},{\rm{Ker}}\e
N_{v})^{\e G_{\be\kappa(\be v\be)}}\e\right]}.\qed
$$
\end{lemma}

\smallskip

\begin{remark} For $v\notin S$, set $w=w_{v}$. If $T$ splits over $K$, then
$$
H^{1}(I_{\overline{w}},{\rm{Ker}}\e
N_{v})=\text{Hom}(I_{\overline{w}},\krn N_{v})=0
$$
since $I_{\overline{w}}$ is torsion and $\krn N_{v}$ is a
torsion-free abelian group with trivial $I_{\overline{w}}\,$-action.
Thus, the inflation-restriction exact sequence shows that
$H^{1}(I_{\overline{v}},{\rm{Ker}}\e N_{v})=H^{1}(I_{w},{\rm{Ker}}\e
N_{v})$ and $G_{k(v)}$ acts on this group through the quotient
$G(w)=G_{k(v)}/\le G_{k(w)}$. Thus, in this case, the formula of the
lemma has the simpler form $q(\delta_{v})=e_{\le v}^{\le
d_{v}}\big/\big[\e H^{1}(I_{w_{v}},{\rm{Ker}}\e N_{v})^{\e
G(w_{v})}\e\big]$.
\end{remark}

\smallskip
Now set
\begin{equation}
H^{1}\ng\big(G,\tto\be(\ut)\big)^{\e\prime}=\textrm{Ker}\ng\left[\e
H^{1}\ng\big(G,\tto\be(\ut)\big)\ra H^{1}(G,T(K))\e\right],
\end{equation}
where the map involved is induced by the natural map
$\tto\be(\ut)\ra T(K)$. Further, recall the map
$$
\gamma\colon T(F)/\e{\s T}^{\e\circ}\lbe(U)\ra
\big(T(K)/\e\tto\be\big(\ut\big)\big)^{G}
$$
from diagram (9).

\begin{lemma} We have
$$
q(\gamma)=\frac{\left[\e
H^{1}\ng\big(G,\tto\be(\ut)\big)^{\e\prime}\e\right]} {\left[\e
\tto\lbe\big(\ut\big)^{G}\colon {\s T}^{\circ}\be(U)\e\right]}.
$$
\end{lemma}
\begin{proof} The commutativity of the exact diagram
$$
\xymatrix{0\ar[r]&{\s T}^{\circ}\be(U)\ar[d]\ar[r]&
T(F)\ar[r]\ar@{=}[d]& T(F)/\e{\s T}^{\e\circ}\be(U)\ar[d]^{\gamma}
\ar[r]& 0 &\\
0\ar[r] &\tto\lbe\big(\ut\big)^{G}\ar[r] & T(F)\ar[r] &
\big(T(K)/\e\tto\be\big(\ut\big)\big)^{G}\ar@{->>}[r]&
H^{1}\ng\big(G,\tto\be(\ut)\big)^{\e\prime} }
$$
yields canonical isomorphisms
$$
\textrm{Ker}(\gamma)=\tto\lbe\big(\ut\big)^{G}\be\big/\e {\s
T}^{\le\circ}\be(U)
$$
and
$$
\textrm{Coker}(\gamma)= H^{1}\ng\big(G,\tto\be(\ut)\big)^{\e\prime}.
$$
The lemma is now clear.
\end{proof}

We can now state our generalization of Chevalley's ambiguous class
number formula.

\begin{teorema} We have

$$
\frac{\left[\e C_{\e T,K,S}^{\e G}\e\right]}{ \left[\e C_{\e
T,F,S}\e\right]}=\frac{\left[\e H^{1}\be\big(G,
T(K)/\e\tto\be\big(\ut\big)\big)^{\e\prime}\e\right]\displaystyle
\prod_{\e v\notin S} q( \e\delta_{v}\e)}{ \left[
H^{1}\ng\big(G,\tto\be(\ut)\big)^{\e\prime}\e\right]\left[\e
\tto\lbe\big(\ut\e\big)^{G}\be\colon\be {\s
T}^{\le\circ}\be(U)\e\right]^{-1}} ,
$$
where $H^{1}\be\big(G, T(K)/\e\tto\be\big(\ut\big)\big)^{\e\prime}$
and $H^{1}\ng\big(G,\tto\be(\ut)\big)^{\e\prime}$ are given by (6)
and (11), respectively, and, for each $v\notin S$,
$q(\e\delta_{v}\e)$ is given by Lemma 3.5.
\end{teorema}
\begin{proof} This is immediate from (10) and Lemma 3.6.
\end{proof}

\smallskip

\smallskip

The next formula is closer to Chevalley's original formula than that
of the theorem.

\begin{corollary} Assume that $T_{K}$ has good reduction over $\ut$,
i.e., that $S_{K}$ contains all primes of bad reduction for $T_{K}$.
Then

$$
\frac{\left[\e C_{\e K,T,S}^{\e G}\e\right]}{ \left[\e C_{\e
T,F,S}\e\right]}=\frac{\left[\e H^{1}\be\big(G, (X^{G_{\be
K}}\be)^{\le\vee}\!\otimes(K^{*}\be/\e{\s O}_{K,S}^{\e
*})\big)^{\e\prime}\e\right]\displaystyle\prod_{\e v\notin S} q(
\e\delta_{v}\e)}{\left[H^{1}\ng\big(G,(X^{G_{\be
K}}\be)^{\le\vee}\!\otimes{\s O}_{K,S}^{\e
*}\big)^{\e\prime}\e\right]\left[\e((X^{G_{\be
K}}\be)^{\le\vee}\!\otimes{\s O}_{K,S}^{\e *})^{\e G}\be\colon\be
{\s T}^{\le\circ}\lbe(U)\e\right]^{-1}}\,.
$$
\end{corollary}
\begin{proof} By hypothesis $\tto$ is a torus over $\ut$, whence
$$
\tto=\s Hom_{\e\ut}(X\be\big(\e\tto\big),\bg_{m})
$$
as sheaves on $\ut_{\e\text{\'et}}$, where $X\be\big(\e\tto\big)$ is
the sheaf of characters of $\tto$. Now, since $X\be\big(\e\tto\big)$
is an \'etale locally constant $\bz$-constructible sheaf on $\ut$,
$X\be\big(\e\tto\big)\be\big(\e\ut\e\big)=X^{G_{\be K}}$ (see, e.g.,
[12, p.156]). Consequently,
$$
\tto\be\big(\ut\e\big)=\text{Hom}(X^{G_{\be K}}\be,{\s O}_{K,S}^{\e
*})=(X^{G_{\be K}}\be)^{\le\vee}\!\otimes{\s O}_{K,S}^{\e *}.
$$
The corollary is now immediate.
\end{proof}

\smallskip

\begin{remark} The corollary is valid for any set $S$
verifying our general assumptions if $T$ splits over $K$. In this
case $X^{G_{\be K}}=X$ and, using Remark 3.2, the formula of the
corollary has the following simpler form:
$$
\frac{\left[\e C_{\e K,T,S}^{\e G}\e\right]}{ \left[\e C_{\e
T,F,S}\e\right]}=\frac{\left[\e H^{1}\be\big(G,
X^{\le\vee}\!\otimes(K^{*}\be/\e{\s O}_{K,S}^{\e
*})\big)\e\right]\prod_{\e v\notin S}
q(\delta_{v})}{\left[H^{1}\ng\big(G,X^{\le\vee}\!\otimes{\s
O}_{K,S}^{\e
*}\big)^{\e\prime}\e\right]\left[\e(X^{\le\vee}\!\otimes{\s
O}_{K,S}^{\e *})^{\e G}\be\colon\be {\s
T}^{\le\circ}\lbe(U)\e\right]^{-1}}\,.
$$
\end{remark}

\section{Norm tori}

Define a $G_{\be F}$-module $X$ by the exactness of the sequence
\begin{equation}
0\ra\bz\overset{\varepsilon}\longrightarrow\bz\be[G]\ra X\ra 0,
\end{equation}
where $\varepsilon$ is given by $\varepsilon(1)=\sum_{\e\sigma\in
G}\sigma$. This exact sequence induces an exact sequence of $F$-tori
\begin{equation}
0\ra T\ra R_{K/F}(\bg_{m,K})\overset{N}\longrightarrow\bg_{m,F}\ra
0,
\end{equation}
where $R_{K/F}(\bg_{m,K}\be)$ is the Weil restriction of
$\bg_{m,K}$, $N$ is induced by the norm map $K\ra F$ and
$T=R_{K/F}^{\e(1)}(\bg_{m,K})$ is the corresponding norm torus. Note
that $T_{K}=\bg_{m,\le K}^{\le n-1}$, where $n=[K\colon\! F\e]$. For
any $v\notin S$, let $w=w_{v}$ be the prime of $K$ lying above $v$
fixed previously. We will write $f_{v}$ for the residue class degree
$[\kappa(w)\colon\!\kappa(v)]$ and $G_{w}^{\text{ab}}$ for the
largest abelian quotient of $G_{w}$, i.e.,
$G_{w}^{\text{ab}}=G_{w}/G_{w}^{\e\prime}$, where $G_{w}^{\e\prime}$
is the commutator subgroup of $G_{w}$. Note that $G(w)$ is a cyclic
group of order $f_{v}$ and $[G_{w}]=e_{v}f_{v}$. Further, since
$G_{w}/I_{w}=G(w)$ is abelian, $I_{w}$ contains $G_{w}^{\e\prime}$.

The dimension $d_{v}$ of the largest split subtorus of $T_{F_{v}}$
is the rank of $X^{G_{w}}$. Since $H^{\e 1}(G_{w},\bz)=0$ and
$\bz\be[G]$ is a free (right) $\bz[G_{w}]$-module of rank $[G\colon
G_{w}]$, (12) shows that $d_{v}=[G\colon G_{w}]-1$. We will now
compute $q(\delta_{v})=e_{\le v}^{\le d_{v}}\big/\big[\e
H^{1}(I_{w},{\rm{Ker}}\e N_{v})^{\e G(w)}\e\big]$ (see Remark 3.6).
Recall the norm element $N_{I_{w}}=\sum_{\e\tau\in
I_{w}}\!\tau\in\bz\be[G]$. Since $\bz\be[G]=\bz[I_{w}]^{[G\colon\!
I_{w}]}$, the multiplication-by-$N_{I_{w}}$ map
$\bz\be[G]\ra\bz\be[G]^{I_{w}}$ is surjective. Further, since $H^{\e
1}(I_{w},\bz)=0$, the canonical map $\bz\be[G]^{I_{w}}\ra X^{I_{w}}$
is surjective as well. It follows that $N_{v}\colon X\ra X^{I_{w}}$
is surjective. Consequently, we have a natural exact sequence of
$I_{w}$-modules
$$
0\ra\krn N_{v}\ra X\overset{N_{v}}\longrightarrow X^{I_{w}}\ra 0.
$$
Taking $I_{w}$-cohomology of the above exact sequence, we obtain a
natural exact sequence of $G(w)$-modules
\begin{equation}
0\ra X^{I_{w}}\!/e_{v}\ra H^{\e 1}(I_{w},\krn N_{v})\ra H^{\e
1}(I_{w},X)\ra 0
\end{equation}
and therefore an exact sequence of abelian groups
\begin{equation}
\begin{array}{rcl}
H^{\e 1}(G(w),X^{I_{w}}\!/e_{v})^{D}\overset{{\partial}^{\le
D}}\longrightarrow \left(H^{\e 1}(I_{w},X)^{G(w)}\right)^{\be
D}\!\!\!&\ra&\!\!\!\left(H^{\e 1}(I_{w},\krn
N_{v})^{G(w)}\right)^{D}\\\\
&\ra &\left((X^{I_{w}}\!/e_{v})^{G(w)}\right)^{\be D}\ra 0,
\end{array}
\end{equation}
where $\partial\,\colon\be H^{\e 1}(I_{w},X)^{G(w)}\ra H^{\e
1}(G(w),X^{I_{w}}\!/e_{v})$ is the ``connecting homomorphism"
appearing in the long $G(w)$-cohomology sequence arising from (14)
(see, e.g., [1, end of \S 2, p.97] for a general description of
$\partial$).

\begin{lemma} There exists an isomorphism
$$
\left(H^{\e 1}(I_{w},X)^{G(w)}\right)^{\be
D}=I_{w}/G_{w}^{\e\prime}.
$$
\end{lemma}
\begin{proof} We will need the dual of the well-known
inflation-restriction-transgression exact sequence [19, p.51],
namely
\begin{equation}
\begin{array}{rcl}
H^{\e 2}\big(G(w),X^{I_{w}}\be\big)^{D}\ra\left(H^{\e
1}(I_{w},X)^{G(w)}\right)^{D}&\ra & H^{\e 1}(G_{w},X)^{D}\\
&\ra & H^{\e 1}\big(G(w),X^{I_{w}}\big)^{D}\ra 0.
\end{array}
\end{equation}
Since $H^{\e 1}(I_{w},\bz)=0$, the exact sequence (12) induces an
exact sequence of $G(w)$-modules
$$
0\ra\bz\ra\bz\be[G(w)]^{[G\colon\! G_{w}]}\ra X^{I_{w}}\ra 0.
$$
Taking $G(w)$-cohomology of the preceding exact sequence and using
Shapiro's lemma together with [4, p.250] (along with the periodicity
of the cohomology of cyclic groups), we obtain isomorphisms
$$
H^{\e 1}(G(w),X^{I_{w}})^{D}=H^{\e 2}(G(w),\bz)^{D}=G(w)
$$
and
$$
H^{\e 2}(G(w),X^{I_{w}})^{D}=H^{\e 3}(G(w),\bz)^{D}=H^{\e
1}(G(w),\bz)^{D}=0.
$$
On the other hand, taking $G_{w}$-cohomology of the exact sequence
of $G_{w}$-modules
$$
0\ra\bz\ra\bz\be[G_{w}]^{[G\colon\! G_{w}]}\ra X\ra 0
$$
we obtain, using Shapiro's lemma and [4, p.250] again, isomorphisms
$$
H^{\e 1}(G_{w},X)^{D}=H^{\e 2}(G_{w},\bz)^{D}=G_{\be
w}^{\e\text{ab}}.
$$
Thus (16) is isomorphic to an exact sequence
$$
0\ra\left(H^{\e 1}(I_{w},X)^{G(w)}\right)^{D}\ra G_{\be
w}^{\e\text{ab}}\ra G(w)\ra 0.
$$
It is not difficult to check that the map $G_{\be
w}^{\e\text{ab}}\ra G(w)$ appearing above is the canonical
projection map $G_{\be w}/G_{\be w}^{\e\prime}\ra
G_{w}/I_{w}\e$\footnote{ This is a bit tedious since it involves
checking the commutativity of a certain diagram whose horizontal
arrows are duals of inflation maps and whose vertical maps are duals
of composites of connecting homomorphisms of the form $H^{\e 1}\ra
H^{\e 2}$ with integral duality isomorphims.}. The lemma is now
immediate.
\end{proof}

\begin{lemma} $\big[\le(X^{I_{w}}\!/e_{v})^{G(w)}\e\big]=e_{v}^{d_{v}}
(f_{v},e_{v})$. Further, there exists an isomorphism
$$
H^{\e 1}(G(w),X^{I_{w}}\!/e_{v})^{D}=G(w)_{e_{v}}.
$$
\end{lemma}
\begin{proof} Since: $\bz\be[G]=
\bz\be[G_{w}]^{[G\colon\! G_{w}]}$ as $G_{w}$-modules,
$\bz\be[G_{w}]^{I_{w}}=\bz\be[G(w)]$ canonically and
$\text{Tor}_{1}^{\bz}(X^{I_{w}},\bz\!/e_{v})=0$, we see that (12)
yields an exact sequence of $G(w)$-modules
$$
0\ra\bz\!/e_{v}\ra(\bz\be[G(w)]\otimes\bz\!/e_{v})^{[G\colon\!
G_{w}]}\ra X^{I_{w}}\!/e_{v}\ra 0.
$$
We conclude, using Shapiro's lemma, that there exists an exact
sequence
$$
0\ra\bz\!/e_{v}\ra(\bz\!/e_{v})^{[G\colon\! G_{w}]}\ra
(X^{I_{w}}\!/e_{v})^{G(w)}\ra\text{Hom}(G(w),\bz\!/e_{v})\ra 0.
$$
The first assertion of the lemma is now clear. On the other hand,
there exist isomorphisms
$$\begin{array}{rcl}
H^{\e 1}(G(w),X^{I_{w}}\!/e_{v})^{D}&=&H^{\e
2}(G(w),\bz\!/e_{v})^{D}= (H^{\e 2}(G(w),\bz)/e_{v})^{D}\\
&=&(H^{\e 2}(G(w),\bz)^{D})_{e_{v}}=G(w)_{e_{v}}
\end{array}
$$
(to obtain the second equality, take $G(w)$-cohomology of the
sequence $0\ra\bz\overset{e_{v}}\longrightarrow\bz\ra\bz\be/e_{v}\ra
0$). This completes the proof.
\end{proof}

We now return to the exact sequence (15). The map labeled
$\partial^{\le D}$ there corresponds, under the isomorphisms of
Lemmas 4.1 and 4.2, to the connecting homomorphism $c\colon
G(w)_{e_{v}}\ra
I_{w}/G_{w}^{\e\prime}=(I_{w}/G_{w}^{\e\prime})/e_{v}$ appearing in
the exact sequence
$$
0\ra I_{w}/G_{w}^{\e\prime}\ra(G_{\be w}^{\e\text{ab}})_{e_{v}}\ra
G(w)_{e_{v}}\overset{c}\longrightarrow I_{w}/G_{w}^{\e\prime}\,,
$$
which is the beginning of the exact sequence (3) (with $m=e_{v}$)
associated to the short exact sequence $0\ra
I_{w}/G_{w}^{\e\prime}\ra G_{\be w}^{\e\text{ab}}\ra G(w)\ra 0$.
Therefore
$$
\left[\cok\!\left(\partial^{\le D}\right)\right]=\frac{\left[(G_{\be
w}^{\e\text{ab}})_{e_{v}}\right]}{(f_{v},e_{v})}.
$$
Now the exactness of (15) together with the first assertion of Lemma
4.2 yield the identity
$$
\left[\le H^{\e 1}(I_{w},\krn
N_{v})^{G(w)}\le\right]=e_{v}^{d_{v}}\!\be \left[(G_{\be
w}^{\e\text{ab}})_{e_{v}}\right].
$$
Thus, by Remark 3.6, the following (simple) formula holds.

\begin{lemma} $q(\delta_{v})=\left[(G_{\be
w_{v}}^{\e{\rm{ab}}})_{e_{v}}\right]^{-1}.\qed $
\end{lemma}

\smallskip

We now compute the global factors entering in the formula of Remark
3.10.

Since $\text{Tor}_{1}^{\e\bz}\be(\bz,-)=0$, there exists an exact
sequence of $G$-modules
$$
0\ra X^{\le\vee}\!\otimes(K^{*}\be/\e{\s O}_{K,S}^{\e
*})\ra\bz\be[G]\otimes(K^{*}\be/\e{\s O}_{K,S}^{\e *})\ra
K^{*}\be/\e{\s O}_{K,S}^{\e *}\ra 0
$$
which is induced by the $\bz$-linear dual of (12). Now, taking Tate
cohomology $\widehat{H}^{\le r}$ of the above exact sequence and
using Shapiro's lemma, we obtain an isomorphism
\begin{equation}
H^{1}\be\big(G, X^{\le\vee}\!\otimes(K^{*}\be/\e{\s O}_{K,S}^{\e
*})\big)=\widehat{H}^{\e 0}(G,K^{*}\be/\e{\s O}_{K,S}^{\e *}).
\end{equation}

Arguing similarly, we obtain an isomorphism
$$
H^{1}\be\big(G, X^{\le\vee}\!\otimes {\s O}_{K,S}^{\e *}\big)=
\widehat{H}^{\e 0}(G,{\s O}_{K,S}^{\e *})={\s O}_{F,S}^{\e *}\big/\e
N_{K/F}({\s O}_{K,S}^{\e *}).
$$
On the other hand, (13) shows that $H^{\e 1}(G,T(K))=F^{\e *}\be/\e
N_{K/F}(K^{\e *})$. Consequently, the map $H^{1}\be\big(G,
X^{\le\vee}\otimes\e {\s O}_{K,S}^{\e *}\big)\ra H^{\e 1}(G,T(K))$
corresponds to the canonical map
$$
{\s O}_{F,S}^{\e *}\big/N_{K/F}({\s O}_{K,S}^{\e *})\ra F^{\e
*}\be/\e N_{K/F}(K^{\e *}).
$$
Thus, we have a canonical isomorphism
$$
H^{1}\be\big(G, X^{\le\vee}\!\otimes {\s O}_{K,S}^{\e
*}\big)^{\e\prime}=\left({\s O}_{F,S}^{\e *}\cap N_{K/F}(K^{\e
*})\right)\lbe\big/N_{K/F}({\s O}_{K,S}^{\e *}).
$$
Setting
\begin{equation}
W_{F,S}={\s O}_{F,\le S}^{\e *}\cap N_{K/F}(K^{\e *})
\end{equation}
we conclude that
\begin{equation}
\left[\e H^{1}\be\big(G, X^{\le\vee}\!\otimes {\s O}_{K,S}^{\e
*}\big)^{\e\prime}\e\right]=\big[\e\widehat{H}^{\e 0}(G,{\s
O}_{K,S}^{\e *})\e\big]\big/\be\left[\e{\s O}_{F,S}^{\e
*}\e\colon\lbe W_{\be F,\le S}\e\right].
\end{equation}

Next, (12) induces an exact sequence of $G$-modules
$$
0\ra X^{\le\vee}\!\otimes{\s O}_{K,S}^{\e *}\ra\bz\be[G]\otimes{\s
O}_{K,S}^{\e *}\ra{\s O}_{K,S}^{\e *}\ra 0,
$$
where the right-hand map is induced by the $\bz$-linear map
$\bz\be[G]\!\!\ra\bz,\,\sigma\mapsto 1$. The preceding exact
sequence induces an exact sequence
$$
0\ra (X^{\le\vee}\!\otimes{\s O}_{K,S}^{\e *})^{\le G}\ra{\s
O}_{K,S}^{\e *}\overset{N_{{\s O}}}\longrightarrow {\s O}_{F,S}^{\e
*},
$$
where $N_{{\s O}}$ is the restriction of $N_{K/F}$ to ${\s
O}_{K,S}^{\e *}$. We conclude that
\begin{equation}
\left[\e(X^{\le\vee}\!\otimes{\s O}_{K,S}^{\e *})^{\e G}\colon\be
{\s T}^{\le\circ}\lbe(U)\e\right]=\left[\,\krn(N_{{\s
O}})\,\colon\lbe {\s T}^{\le\circ}\lbe(U)\e\right]
\end{equation}

\begin{remarks}
(a) If (13) extends to an exact sequence
$$
0\ra\sto\ra R_{\e{\s O}_{K,S}/{\s O}_{F,S}}\be(\bg_{m,{\s
O}_{K,S}}\be)\overset{N_{\s O}}\longrightarrow \bg_{m,{\s O}_{F,S}}
$$
of identity components of N\'eron-Raynaud models (cf. [15, Lemma
3.1]), then the index (20) is clearly 1.

(b) The proofs of Lemma 3.7 and Corollary 3.9, together with diagram
(9) (see also Remark 3.10), show that the index
$$
\left[\e(X^{\le\vee}\!\otimes{\s O}_{K,S}^{\e *})^{\e G}\be\colon\be
{\s T}^{\le\circ}\lbe(U)\e\right]=\left[\e
\tto\lbe\big(\ut\e\big)^{\lbe G}\be\colon\be {\s
T}^{\le\circ}\be(U)\e\right]=\left[\e\krn(\gamma)\e\right]
$$
divides $\prod_{\, v\notin
S}\left[\e\Phi_{v}(\kappa(v))_{\e\text{tors}}\e\right]$. But [23,
Corollary 2.18, p.175] and Lemma 4.1 above show that
$\left[\e\Phi_{v}(\kappa(v))_{\e\text{tors}}\e\right]=\left[\e
I_{w_{v}}\lbe\colon\lbe G_{\lbe w_{v}}^{\e\prime}\right]$ for any
$v$. Consequently, $\left[\,\krn(N_{{\s O}})\,\colon\lbe {\s
T}^{\le\circ}\lbe(U)\e\right]$ divides $\prod_{\e v\notin S}\be
\left[\e I_{w_{v}}\lbe\colon\lbe G_{\lbe w_{v}}^{\e\prime}\right]$,
i.e., the latter integer is an upper bound for (20).

\smallskip

\end{remarks}

\smallskip

Combining Remark 3.10, Lemma 4.3 and formulas (17)-(20), we obtain
{\it Chevalley's ambiguous class number formula for a norm torus}:

\begin{teorema} Let $K/F$ be a Galois extension of global
fields of degree $n$ and Galois group $G$. Let
$T=R_{K/F}^{\e(1)}(\bg_{m,K})$ be the corresponding norm torus. Then

$$
\frac{\left[\e (C_{K,\le S}^{\e n-1})^{\e G}\e\right]}{h_{\e
T,F,S}}=\frac{\big[\e\widehat{H}^{\e 0}(G,K^{*}\be/\e{\s
O}_{K,S}^{\e *})\e\big]\be\left[\e{\s O}_{F,S}^{\e *}\colon\be
W_{\be F,\le S}\e\right]\be\left[\,{\rm{Ker}}(N_{{\s O}})\colon\be
{\s T}^{\le\circ}\lbe(U)\e\right]}{\big[\e\widehat{H}^{\e 0}(G,{\s
O}_{K,S}^{\e *})\e\big]\displaystyle\prod_{v\notin S}\left[(G_{\be
w_{v}}^{\e{\rm{ab}}})_{e_{v}}\right]}\,,
$$
where $W_{\be F,\le S}$ is the group (18).$\qed$
\end{teorema}

Next we need a definition. Let $G$ be any finite group and let $M$
be a left $G$-module such that both $\widehat{H}^{\e 0}(G,M)$ and
$H^{1}(G,M)$ are finite. We define
\begin{equation}
h(G,M)=\frac{\big[\e\widehat{H}^{\e 0}(G,M)\e\big]}{\big[\e
H^{1}(G,M)\e\big]}.
\end{equation}
When $G$ is {\it cyclic}, $h(G,M)$ is known as the {\it Herbrand
quotient} of $M$. We now divide Chevalley's original formula
(Theorem 1.1 of the Introduction) by the formula in the preceding
theorem and obtain$\e$\footnote{$\e$ We embed $C_{\lbe K,\le S}$ in
$C_{\lbe K,\le S}^{\e n-1}$ diagonally.}
\begin{corollary} With the notations of the theorem, we
have
$$
\frac{h_{\e T,F,S}}{h_{\e F,S}}=\frac{\left[\e \left(C_{\lbe K,\le
S}^{\e n-1}\right)^{\be G}\be\colon\be C_{\lbe K,\le S}^{\e
G}\e\right]h(G,{\s O}_{K,S}^{\e *})\prod_{\e v\notin
S}\be\left\{\left[(G_{\be w_{v}}^{\e{\rm{ab}}})_{e_{v}}\right]\be
e_{v}\right\}}{\left[\e{\s O}_{F,S}^{\e *}\colon\be W_{\be F,\le
S}\e\right] h(G,K^{*}\be/\e{\s O}_{K,S}^{\e
*})\left[\,{\rm{Ker}}(N_{{\s O}})\colon\be {\s
T}^{\le\circ}\lbe(U)\e\right],}
$$
where $h(G,{\s O}_{K,S}^{\e *})$ and $\, h(G,K^{*}\be/\e{\s
O}_{K,S}^{\e *})$ are given by (21).\qed
\end{corollary}

Assume now that $K/F$ is a  {\it cyclic} extension of degree $n$.
The exact sequence $1\ra {\s O}_{K,S}^{\e *}\ra K^{*}\ra
K^{*}\be/\e{\s O}_{K,S}^{\e *}\ra 1 $ and Hilbert's Theorem 90 show
that
$$\begin{array}{rcl}
h(G,K^{*}\be/\e{\s O}_{\lbe K,S}^{\e *})&=& h(G,K^{\e *})/h(G,{\s
O}_{K,S}^{\e *})\\\\
&=&[\e F^{\le *}\colon\! N_{K/F}(K^{*})\e]\le/\le h(G,{\s
O}_{K,S}^{\e *})
\end{array}
$$
(see, e.g., [1, Proposition 10, p.109]). On the other hand, we have
the well-known formula
$$
h(G,{\s O}_{K,S}^{\e *})=\frac{1}{n}\displaystyle\prod_{v\in
S}\,[\le K_{w_{v}}\be\colon\! F_{v}\e]
$$
(see [20, proof of Theorem 8.3, p.179]). Finally, $\left[(G_{\be
w_{v}}^{\e{\rm{ab}}})_{e_{v}}\right]=e_{v}$ for any $v\notin S$.
Using the above, we deduce from Corollary 4.6 the following formula.

\begin{corollary} Let $K/F$ be a cyclic extension of global
fields of degree $n$ and Galois group $G$. Let
$T=R_{K/F}^{\e(1)}(\bg_{m,K})$ be the corresponding norm torus. Then
$$
\frac{h_{\e T,F,S}}{h_{\e F,S}}=\frac{\left[\e \left(C_{\lbe K,\le
S}^{\e n-1}\right)^{\be G}\be\colon\be C_{\lbe K,\le S}^{\e
G}\e\right]\,\prod_{\e v\in S}\e [\le K_{w_{v}}\be\colon\!
F_{v}\e]^{\le 2}\,\prod_{\,v\notin S}e_{v}^{\le 2}}{n^{2}\left[\e{\s
O}_{F,S}^{\e *}\colon\be W_{\be F,\le S}\e\right]\be \left[\e F^{\e
*}\colon\! N_{K/F}(K^{\le *})\e\right]\be\left[\,{\rm{Ker}}(N_{{\s
O}})\colon\be {\s T}^{\le\circ}\lbe(U)\e\right].}
$$
\end{corollary}

\medskip

Assume now that $K/F$ is a {\it quadratic} extension, i.e., $n=2$.
Clearly,
$$
\left[\e \left(C_{\lbe K,\le S}^{\e n-1}\right)^{\be G}\be\colon\be
C_{\lbe K,\le S}^{\e G}\e\right]=1.
$$
Let $\mu$ be the number of primes in $S$ which do not split in $K$,
i.e., those $v\in S$ such that $[\le K_{w_{v}}\be\colon\!
F_{v}\e]=2$, and let $\nu$ be the number of primes of $F$ not in $S$
which ramify in $K$. Then
$$
\prod_{\e v\in S}\e[\le K_{w_{v}}\be\colon\! F_{v}\e]^{\le
2}\,\prod_{v\notin S}e_{v}^{\le 2}\,=\,4^{\le\mu+\nu}.
$$
Thus Corollary 4.7 yields

\begin{corollary} Let $K/F$ be a {\rm{quadratic}} extension of global
fields and let $T=R_{K/F}^{\e(1)}(\bg_{m,K})$ be the corresponding
norm torus. Then
$$
\left[\e{\s O}_{F,S}^{\e *}\colon\! W_{\be F,\le S}\e\right]\be
\left[\e F^{\e *}\be\colon\! N_{K/F}(K^{\le
*})\e\right]\be\left[\,{\rm{Ker}}(N_{{\s O}})\colon\be {\s
T}^{\le\circ}\lbe(U)\e\right]\cdot h_{\e T,F,S}=4^{\e\mu+\nu-1}h_{\e
F,\le S},
$$
where $\mu$ is the number of primes in $S$ which do not split in $K$
and $\nu$ is the number of primes of $F$ which ramify in $K$ but are
not in $S$. In particular, $h_{\e T,F,S}$ divides
$4^{\e\mu+\nu-1}h_{\e F,\le S}\,$ if $\,\mu+\nu\geq 1.\qed$
\end{corollary}

\begin{remark} Another interesting example, which we hope to
discuss at length in [8], involves the quotient torus $\mathbb
S=R_{\e K/F}(\bg_{m,K})/\bg_{m,F}$. This torus is dual to
$R_{K/F}^{\e(1)}(\bg_{m,K})$, and is {\it isomorphic} to it if $K/F$
is cyclic (see [11, Lemma 4.1, p.201]). In particular, the formula
of Corollary 4.7 holds true with $\mathbb S$ in place of $T$ (if
$K/F$ is cyclic).
\end{remark}

\section{Concluding remarks}

In [14, Theorem, p.135] M.Morishita, generalizing work of T.Ono,
obtained a formula for the class number of the norm
``torus"\footnote{ We use quotation marks because, in general, $\s
T^{\e\prime}$ is not a torus (it {\it is} a torus if ${\s
O}_{K,S}/{\s O}_{F,S}$ is \'etale).} $\s
T^{\e\prime}=R^{\,(1)}_{\e{\s O}_{K,S}/{\s O}_{F,S}}$ associated to
${\s O}_{K,S}/{\s O}_{F,S}$, i.e., $\s T^{\e\prime}$ is the kernel
of the norm map
$$
N_{\s O}\colon R_{\e{\s O}_{K,S}/{\s O}_{F,S}}\be(\bg_{m,{\s
O}_{K,S}}\be)\ra\bg_{m,{\s O}_{F,S}}.
$$
Since in general the exact sequence (13) above does not extend in
the sense of Remark 4.4(a) (not even in the tamely ramified case.
Cf. [18, Lemma 6.7, p.28]), $\sto$ and $\s T^{\e\prime}$ need not
coincide (indeed, $\s T^{\e\prime}$ need not have connected fibers).
It follows that Morishita's formula and our own (see, e.g.,
Corollary 4.6) are not immediately comparable. It is an interesting
problem, however, to explore the deeper connections that certainly
exist between the approach of [14] and that of this paper. We hope
to carry out this project in a not-too-distant future.

\end{document}